\documentclass[10pt]{article}

\usepackage{amsmath}
\usepackage{amssymb}
\usepackage{amsfonts}
\usepackage{amsthm}
\usepackage{latexsym}
\usepackage{graphicx}

\newtheorem{definition}{Definition}[section]
\newtheorem{theorem}[definition]{Theorem}
\newtheorem{lemma}[definition]{Lemma}

\newcommand{\variable}{\mathsf{Var}}

\newcommand{\logick}{\mathbf{K}}
\newcommand{\propgl}{\mathbf{GL}}
\newcommand{\predgl}{\mathbf{QGL}}
\newcommand{\mppa}{\mathbf{QPL}(\logicpa)}
\newcommand{\predlogic}[1]{\mathbf{MQ}_{\mathcal{#1}}}
\newcommand{\proplogic}{\mathbf{MP}}

\newcommand{\logicpa}{\mathsf{PA}}
\newcommand{\lk}{\mathsf{LK}}
\newcommand{\glsys}{\mathsf{NQGL}}
\newcommand{\glsysc}{\mathsf{NQGL}^{-}}

\newcommand{\framefi}{\mathfrak{FI}}
\newcommand{\framecw}{\mathfrak{CW}}
\newcommand{\framebl}{\mathfrak{BL}}

\newcommand{\kmodel}[1]{{#1}}

\newcommand{\iand}{\bigwedge}

\newcommand{\thn}{\ \Rightarrow\ }
\newcommand{\eq}{\ \Leftrightarrow\ }

\newcommand{\yy}{\rightarrow}

\newcommand{\gldb}{\db_{\glsys}}
\newcommand{\gldbc}{\db_{\glsysc}}
\newcommand{\db}{\vdash}
\newcommand{\vl}{\models}

\newcommand{\gm}{\Gamma}
\newcommand{\dl}{\Delta}
\newcommand{\ld}{\Lambda}
\newcommand{\xx}{\Xi}

\newcommand{\predb}{\mathrm{BF}}
\newcommand{\barcan}{\forall x\Box\phi\supset\Box\forall x\phi}

\title{A cut-free proof system for a predicate extension of the logic of provability}
\author{Yoshihito Tanaka\\ 
Kyushu Sangyo University\\ 
{\tt ytanaka@ip.kyusan-u.ac.jp}}

\date{}

\begin{document}

\maketitle

\begin{abstract}
In this paper, we introduce a proof system $\glsys$
for a Kripke complete predicate extension of the logic 
$\propgl$, that is, the logic of provability, which is defined by $\logick$ and 
the L\"{o}b formula $\Box(\Box p\supset p)\supset \Box p$.  
$\glsys$ is a modal extension of Gentzen's sequent calculus $\lk$. 
Although the propositional fragment of $\glsys$ axiomatizes $\propgl$, 
it does not have the  L\"{o}b formula as its axiom. 
Instead, it has a non-compact rule, 
that is, a derivation rule with countably many premises. 
We show that $\glsys$ enjoys cut admissibility and is 
complete with respect to the class of Kripke 
frames such that for each world, the supremum of the length of 
the paths from the world is finite. 
\end{abstract}

\section{Introduction}

In this paper, we introduce a cut-free proof system for a 
Kripke complete predicate extension of $\propgl$, where  
$\propgl$ is 
a propositional normal modal logic
defined by $\logick$ and the L\"{o}b formula 
\begin{equation}\label{glaxiom}
\Box(\Box p\supset p)\supset \Box p.   
\end{equation}
$\propgl$ is well-known as the logic of provability, in the sense that 
a propositional modal formula $\phi$ is in $\propgl$ if 
and only if $f(\phi)$ is provable in the Peano arithmetic $\logicpa$
for every arithmetical interpretation $f$ (e.g. \cite{bls93}).

A Kripke frame $(W,R)$ is said to be {\em conversely well-founded}, if there 
exists no countably infinite list 
$(w_{i})_{i\in\mathbb{N}}$ of elements of $W$ which satisfies
$(w_{i},w_{i+1})\in R$ for any $i\in\mathbb{N}$, 
and is said to be {\em of bounded length}, if for any $w\in W$ 
the supremum of the length of the lists 
$w_{0},w_{1},\ldots,w_{n}$ which satisfy $(w_{i},w_{i+1})\in R$ and $w_{0}=w$ 
is finite. 
We write
$\framefi$, $\framebl$, and $\framecw$ for the classes of transitive Kripke frames 
which are finite and irreflexive, of bounded length, and conversely well-founded, respectively.  
For any class $C$ of Kripke frames, 
we write $\proplogic(C)$ and $\predlogic{}(C)$ for 
the sets of propositional modal formulas and 
predicate modal formulas 
which are valid in $C$, respectively. 
It is known (e.g. \cite{bls93}) that 
$$
\propgl=\proplogic(\framefi)=\proplogic(\framecw).
$$
Therefore, 
\begin{equation}\label{eq:propgl}
\propgl=\proplogic(\framebl). 
\end{equation}

However, the situation in predicate extensions of $\propgl$ is 
not so clear. 
Let $\predgl$ be the smallest predicate normal modal logic which includes 
$\propgl$ as its propositional fragment. 
Let 
$\mppa$ be the set of predicate modal formulas defined by  
$$
\mppa
=
\{
\phi\mid
\text{
$\logicpa\db f(\phi)$ for every interpretation $f$
}
\}.
$$

\noindent
It is shown in \cite{mnt84}  
that $\predgl\subsetneqq\predlogic{}(\framecw)$ and 
$\predgl$ is incomplete with respect to any classes of Kripke frames. 
It is also proved in \cite{mnt84} that 
$\predgl\subsetneqq\mppa$, 
that is,   
$\predgl$ is arithmetically incomplete, 
and $\mppa\not\subseteqq\predlogic{}(\framefi)$.
Subsequently, 
\cite{art-dzh90} 
shows that if a closed predicate modal formula $\phi$ is 
not valid in a finite irreflexive Kripke model with finite domains
then there exists an interpretation 
$f$ such that $\logicpa\not\db f(\phi)$.
To summarize these results, we have the following:

\begin{equation}\label{map}
\begin{array}{ccl}
\cline{1-1}
\multicolumn{1}{|c|}{}
&
\subseteqq
&
\predlogic{}(\text{$\framefi$ with finite domains})\\
\multicolumn{1}{|c|}{}
&
&
\hspace{5pt}
\rotatebox[origin=c]{90}{$\subsetneqq$}\\
\multicolumn{1}{|c|}{}
&
\not\subseteqq
&
\predlogic{}(\framefi)\\
\multicolumn{1}{|c|}{}
&
&
\hspace{5pt}
\rotatebox[origin=c]{90}{$\subseteqq$}\\
\multicolumn{1}{|c|}{\mppa}
&
&
\predlogic{}(\framebl)\\
\multicolumn{1}{|c|}{}
&
&
\hspace{5pt}
\rotatebox[origin=c]{90}{$\subseteqq$}\\
\multicolumn{1}{|c|}{}
&
&
\predlogic{}(\framecw)\\
\multicolumn{1}{|c|}{}
&
&
\hspace{5pt}
\rotatebox[origin=c]{90}{$\subsetneqq$}\\
\multicolumn{1}{|c|}{}
&
\supsetneqq
&
\predgl\\
\cline{1-1}
\end{array}.
\end{equation}

On the other hand, 
\cite{yvr01}
introduces a logic $\predgl^{b}$, a predicate extension of $\propgl$, 
in which all occurrences of individual variables 
in a scope of a modal operator are considered to be bound, and 
$$
\Box\phi\yy\Box\forall x\phi
$$
is an axiom schema. 
It is proved in \cite{yvr01} 
that $\predgl^{b}$ is both arithmetically complete and Kripke complete
with respect to $\framefi$, under the above restriction in the construction of formulas.

In \cite{lvn81}, a sequent system for $\propgl$ is introduced, of which 
modal rule is  
\begin{eqnarray} \label{glrule}
\frac
{\Box\gm,\gm, \Box\phi\yy\phi}
{\Box\gm\yy\Box\phi}.
\end{eqnarray}
A proof of the cut-elimination theorem of the system is given 
in \cite{vln83} by a syntactic method, and 
a semantic proof of it is given in \cite{avr84}.  
It is also proved in \cite{avr84}
that the simple predicate extension of the system does not 
admit cut-elimination. 
While a sequent of the above sequent system is defined to be a pair of sets of formulas, 
\cite{gor-rmn08} gives a translation of the argument in \cite{vln83} 
to a sequent system built from multisets. 
A cut-free proof system for $\predgl^{b}$ is introduced in \cite{sch-trl14}.

Though 
none of 
$\predlogic{}(\framecw)$,
$\predlogic{}(\framebl)$, nor
$\predlogic{}(\framefi)$ are 
arithmetically complete as described in (\ref{map}), 
it could be of some interest as a problem of pure modal logic to 
give a cut-free proof system for a Kripke complete predicate extension of $\propgl$ without any 
restriction in the construction of formulas. 
In this paper, we introduce a proof system $\glsys$, 
which is a modal extension of 
Gentzen's sequent calculus $\lk$ 
for predicate logic,
and show the admissibility of the cut-rule and 
Kripke completeness with respect to $\framebl$.
From the Kripke completeness, it follows by (\ref{eq:propgl}) that 
the propositional fragment of $\glsys$ axiomatizes $\propgl$, 
but $\glsys$ 
does not include (\ref{glaxiom}) nor (\ref{glrule}) as an axiom schema or a 
derivation rule, respectively. 
Instead, it 
has a non-compact rule, that is,
a derivation rule with countably many premises. 
In \cite{gld93} and \cite{sgb94}, 
a general theory for model existence theorem for propositional modal logic with 
non-compact rules is given,  
also, in \cite{tnknoncpt}, for their predicate extension with 
Barcan formula
$$
\predb=\barcan.
$$
It follows immediately as a corollary of the main theorem of \cite{tnknoncpt}, 
that the system defined by $\glsys$ and $\predb$ is 
Kripke complete with respect to $\framebl$ with constant domains.  
However, it is shown in \cite{mnt84} that
$\predb$ is not $\logicpa$-valid.    
Therefore, we do not add $\predb$.

The outline of the paper is the following:
In Section \ref{syntax_semantics}, we give basic definitions for 
syntax and semantics. 
In Section \ref{rules}, we introduce the system $\glsys$. 
In Section \ref{consistent_pair}, the notions of finitely consistent pairs and saturated 
pairs are introduced. 
In Section \ref{completeness}, we show Kripke 
completeness of $\glsys$ with respect to $\framebl$, as well as 
the admissibility of the cut-rule.

\section{Preliminaries}\label{syntax_semantics}

The language we consider consists of the 
following symbols:
\begin{enumerate} 
\item
a countable set $\mathcal{V}$ of variables;

\item
$\top$ and $\bot$;

\item logical connectives: 
$\land$,  $\neg$, $\supset$; 

\item
quantifier:
$\forall$;

\item
for each $n\in\mathbb{N}$, 
countably many predicate symbols 
$P$, $Q$, $R$, $\cdots$ of arity $n$;

\item
modal operator
$\Box$.
\end{enumerate}

\noindent
The set $\Phi(\mathcal{V})$ of formulas over
$\mathcal{V}$ is the smallest set which satisfies:

\begin{enumerate}
\item
$\top$ and $\bot$ are in $\Phi(\mathcal{V})$;

\item 
if $P$ is a predicate symbol of arity $n$ and 
$x_{1},\ldots,x_{n}$ are variables in $\mathcal{V}$
then
$P(x_{1},\ldots,x_{n})$ is in $\Phi(\mathcal{V})$;

\item 
if $\phi$ and $\psi$ are in $\Phi(\mathcal{V})$
then $(\phi\land\psi)$ 
and $(\phi\supset\psi)$ are in $\Phi(\mathcal{V})$;

\item 
if $\phi\in\Phi(\mathcal{V})$ 
then $(\neg\phi)$ and 
$(\Box\phi)$  are in $\Phi(\mathcal{V})$;

\item
if $\phi\in\Phi(\mathcal{V})$ 
and $x\in\mathcal{V}$
then $(\forall x\phi)\in\Phi(\mathcal{V})$. 
\end{enumerate}

\noindent
As usual,  $\lor$ and $\exists$ are the duals of $\land$ and $\forall$, respectively.  
The symbol $\Diamond$ is an abbreviation of $\neg\Box\neg$, and 
for each $n\in\mathbb{N}$,  
$\Box^{n}$ and $\Diamond^{n}$ denote $n$-times applications of $\Box$ and $\Diamond$, 
respectively.  
For each set $S$ of formulas, we write $\Box S$ and $\Box^{-1}S$ for the sets
$$
\Box S=\{\Box\phi\mid\phi\in S\},\ \ 
\Box^{-1} S=\{\phi\mid\Box\phi\in S\}
$$
of formulas, respectively. 
For each formula $\phi$, 
we write $\variable(\phi)$ for the set of variables which have some 
free or bound occurrences in $\phi$.  For each set $S$ of formulas,  
$\variable(S)$ denotes the set $\bigcup_{\phi\in S}\variable(\phi)$. 
For each subset $\mathcal{U}$ of $\mathcal{V}$, 
$$
\Phi(\mathcal{U})=\{\phi\in\Phi(\mathcal{V})\mid\variable(\phi)\subseteqq\mathcal{U}\}. 
$$

A {\em Kripke frame} 
is a pair $(W,R)$, 
where $W$ is a non-empty set and 
$R$ is a binary relation on $W$. 
A {\em system of domains} over a frame $F=(W,R)$
is a family $D=(D_{w})_{w\in W}$ of non-empty sets 
such that for all $w_{1}$ and $w_{2}$ in $W$, 
$$
(w_{1},w_{2})\in R \thn D_{w_{1}}\subseteqq D_{w_{2}}. 
$$
A {\em predicate Kripke frame} over $F=(W,R)$
is a triple $(W,R,D)$, where 
$D$ is a system of domains over $F$. 
A {\em Kripke model}
is a four tuple $(W,R,D,I)$, 
where $(W,R,D)$ is a predicate Kripke frame
and
$I$ is a mapping 
called an {\em interpretation}
which maps each pair $(w,P)$, where $w$ is a member of $W$ and 
$P$ is a $n$-ary predicate symbol, to an $n$-ary relation
$I(w,P)\subseteqq (D_{w})^{n}$ over $D_{w}$. 
The relation $\vl$ among a Kripke model $\kmodel{M}=(W,R,D,I)$, 
a world $w\in W$, and 
a closed formula $\phi$ 
is defined inductively as follows:
\begin{enumerate}
\item
$\kmodel{M}, w\vl \top$, 
$\kmodel{M}, w\not\vl \bot$; 
\item 
for any predicate $P$ of arity $n$,\\
\noindent
$\kmodel{M}, w\vl P(d_{1},\ldots,d_{n})$
$\eq$
$(d_{1},\ldots,d_{n})\in I(w,P)$;
\item 
$\kmodel{M}, w\vl\phi\land\psi$ 
$\eq$
$\kmodel{M}, w\vl\phi$ and $\kmodel{M}, w\vl\psi$;
\item 
$\kmodel{M}, w\vl\phi\supset\psi$ 
$\eq$ 
$\kmodel{M},w\not\vl\phi$ or $\kmodel{M}, w\vl\psi$;
\item 
$\kmodel{M}, w\vl\neg\phi$ 
$\eq$ 
$\kmodel{M},w\not\vl\phi$;
\item 
$\kmodel{M}, w\vl\forall x\phi$ 
$\eq$ 
$\kmodel{M}, w\vl \phi[d/x]$ for any $d\in D_{w}$;
\item 
$\kmodel{M}, w\vl\Box\phi$ 
$\eq$
$(w,w')\in R$ implies 
$\kmodel{M}, w'\vl\phi$ for any $w'$ in $W$.
\end{enumerate}

Validity of a non-closed formula 
is defined by the validity of the universal closure of it.  
Let $\phi$ be a formula. 
If every world $w$ in a Kripke model $\kmodel{M}$ satisfies  
$\kmodel{M},w\vl\phi$, 
we write $\kmodel{M}\vl\phi$.  
If every Kripke model $\kmodel{M}$ over a frame $F$
satisfies 
$\kmodel{M}\vl\phi$, 
we write $F\vl \phi$. 
If every $F$ in a class $C$ of Kripke frames
satisfies 
$F\vl\phi$, 
we write $C\vl \phi$. 
The following lemma holds immediately:

\begin{lemma}\label{boundedness}
For any Kripke model $\kmodel{M}=(W,R,D,I)$, the underlying frame 
$(W,R)$ is of bounded length if and only if 
for any $w\in  W$ there exists some $n\in\mathbb{N}$ such that
$M,w\vl\neg\Diamond^{n}\top$. 
\end{lemma}

\section{Non-compact proof system for predicate extension of the logic of 
provability}
\label{rules}

In this section, we introduce a proof system  $\glsys$
for a predicate extension of $\propgl$. 
The proof system $\glsys$ is 
a variant of Gentzen-style sequent calculus.  
A sequent $\gm\yy\dl$ is 
defined to be a pair of finite sets 
$\gm$ and $\dl$ of formulas. 
The axiom schemta of $\glsys$ are 
$p\yy p$, $\yy\top$, $\bot\yy$, and 
the derivation rules of $\glsys$ are the following: 

\begin{description}
\item[Set]
$$
\frac
{\gm\yy\dl}
{\gm'\yy\dl'}
\hspace{10pt}
(\text{{\it where $\gm\subseteqq\gm'$ and  $\dl\subseteqq\dl'$}})
$$

\item[Cut]
$$
\frac
{\gm\yy\dl,\phi \hspace{10pt} \phi,\ld\yy\xx}
{\gm,\ld\yy\dl,\xx}
$$

\item[Conjunction]
$$
\frac
{\gm\yy\dl,\phi \hspace{10pt} \gm\yy\dl,\psi} 
{\gm\yy\dl,\phi\land\psi}
\hspace{20pt} 
\frac
{\phi,\gm\yy\dl} 
{\phi\land\psi,\gm\yy\dl}
\hspace{20pt} 
\frac
{\psi,\gm\yy\dl}
{\phi\land\psi,\gm\yy\dl}
$$

\item[Implication]
$$
\frac
{\phi,\gm\yy\dl,\psi}
{\gm\yy\dl,\phi\supset\psi}
\hspace{20pt}
\frac
{\gm\yy\dl,\phi \hspace{10pt} \psi,\ld\yy\xx} 
{\phi\supset\psi,\gm,\ld\yy\dl,\xx}
$$

\item[Negation]
$$
\frac
{\phi,\gm\yy\dl}
{\gm\yy\dl,\neg\phi}
\hspace{20pt} 
\frac
{\gm\yy\dl,\phi} 
{\neg\phi,\gm\yy\dl}
$$

\item[For all]
$$
\frac
{\gm\yy\dl,\phi[y/x]}
{\gm\yy\dl,\forall x\phi}
\hspace{20pt} 
\frac
{\phi[z/x],\gm\yy\dl} 
{\forall x\phi,\gm\yy\dl}
$$
{\em Here, $y$ is a variable in $\mathcal{V}$
which does not occur in any formulas in the lower sequent, 
and $z$ is any variable in $\mathcal{V}$.}

\item[Box]
$$
\frac
{\Box\gm,\dl\yy\phi}
{\Box\gm,\Box\dl\yy\Box\phi}
$$

\item[Boundedness of length]
$$
\frac
{\gm\yy\dl,\Diamond^{n}\top\hspace{10pt}
(\text{for any $n\in\mathbb{N}$})}
{\gm\yy\dl}
$$
{\em Here, the set of upper sequents is countably infinite. }

\end{description}

\bigskip

For any sequent $\gm\yy\dl$, we write $\gldb\gm\yy\dl$ if it is derivable in 
$\glsys$.  
A formula $\phi$ is said to be derivable in $\glsys$, 
if $\gldb\yy\phi$. If this is the case, we write $\gldb\phi$.
It is easy to see that 
the rule $\mathbf{Box}$ is equivalent to 
$\Box p\supset \Box\Box p$ plus standard necessitation rule 
$$
\frac
{\gm\yy\phi}
{\Box\gm\yy\Box\phi}.
$$

\noindent
The rule {\bf Boundedness of length} denotes that 
\begin{eqnarray}\label{algbarcan}
\iand_{n\in\mathbb{N}}\Diamond^{n}1=0
\end{eqnarray}
holds in the Lindenbaum algebra of the logic defined by $\glsys$. 
Note that if a Boolean algebra with operators satisfies (\ref{algbarcan}), 
the following equation holds in it, either: 
$$
\iand_{n\in\mathbb{N}}\Box\Diamond^{n}1=\Box 0. 
$$

\begin{theorem} \label{soundness}
(Soundness of $\glsys$). 
If $\gldb\phi$, then $\framebl\models\phi$, 
for any formula $\phi$. 
\end{theorem}

\section{Finitely consistent pairs and saturated pairs}\label{consistent_pair}
In this section, we introduce some notions which are used
to show the Kripke completeness and 
the admissibility of the cut-rule. 
We write $\glsysc$ for the cut-free fragment of $\glsys$, and 
$\gldbc\gm\yy\dl$ 
if a sequent $\gm\yy\dl$ is derivable in $\glsysc$.

\begin{definition}
A pair $(S,T)$ of sets of formulas is said to be 
{\em finitely consistent} if 
for any finite sets $S'\subseteqq S$ and $T'\subseteqq T$,
$$
\not\gldbc S'\yy T'. 
$$
\end{definition}

\begin{definition}\label{saturated_pair}
Let $\mathcal{U}$ be a set of variables. 
A finitely consistent pair $(S,T)$ of subsets of $\Phi(\mathcal{U})$
is said to be {\em $\mathcal{U}$-saturated}, 
if the following conditions are satisfied:
\begin{enumerate}
\item
If $\phi_{1}\land\phi_{2}\in S$, then 
$\phi_{1}$, $\phi_{2}\in S$, and 
if $\phi_{1}\land\phi_{2}\in T$, then 
either $\phi_{1}\in T$ or $\phi_{2}\in T$. 

\item
If $\phi_{1}\supset\phi_{2}\in S$, then 
either $\phi_{1}\in T$ or $\phi_{2}\in S$, 
and 
if $\phi_{1}\supset\phi_{2}\in T$, then 
$\phi_{1}\in S$ and $\phi_{2}\in T$. 

\item
If $\neg\phi\in S$, then 
$\phi\in T$, 
and 
if $\neg\phi\in T$, then 
$\phi\in S$.

\item
If $\forall x\phi\in S$, then 
$\phi[z/x]\in S$ for all $z\in\mathcal{U}$, 
and 
if $\forall x\phi\in T$, then 
$\phi[z/x]\in T$ for some $z\in\mathcal{U}$. 

\end{enumerate}
\end{definition}

\begin{definition}
A finitely consistent pair $(S,T)$ of formulas is called a {\em $\propgl$-pair}, 
if $\Box\neg\Diamond^{n}\top\in S$ for some $n\in\mathbb{N}$. 
\end{definition}

\begin{theorem}\label{s-saturated}
Let $\mathcal{U}$ be a coinfinite subset of $\mathcal{V}$. 
Suppose $(S,T)$ is a finitely consistent pair of subsets of $\Phi(\mathcal{U})$. 
Then, 
there exists a coinfinite subset $\mathcal{U}'$ of $\mathcal{V}$ and 
a $\mathcal{U}'$-saturated pair
$(S',T')$ such that $\mathcal{U}\subseteqq\mathcal{U}'$, 
$S\subseteqq S'$, and $T\subseteqq T'$.

\end{theorem}

\begin{proof}
Take a coinfinite subset $\mathcal{W}$ of $\mathcal{V}$ such that 
$\mathcal{U}$ is a coinfinite subset of $\mathcal{W}$. 
Let 
$
(\phi_{n})_{n\in\mathbb{N}}
$ 
be a sequence of formulas of $\Phi(\mathcal{W})$ such that 
each formula of $\Phi(\mathcal{W})$ occurs infinitely many times in it. 
For example, if $(\gamma_{n})_{n\in\mathbb{N}}$ is an enumeration of 
all formulas of $\Phi(\mathcal{W})$, $(\phi_{n})_{n\in\mathbb{N}}$ could be 
$$
\gamma_{0},\ \gamma_{0},\gamma_{1},\ \gamma_{0},\gamma_{1},\gamma_{2},\ 
\gamma_{0},\gamma_{1},\gamma_{2},\gamma_{3},\cdots.
$$

\noindent
Define lists 
$(\mathcal{U}_{n})_{n\in\mathbb{N}}$ and 
$((S_{n},T_{n}))_{n\in\mathbb{N}}$ 
which satisfies the following:
\begin{enumerate}
\item
for every $n\in\mathbb{N}$, 
$\mathcal{U}_{n}$ is a coinfinite subset of $\mathcal{W}$ and 
$\mathcal{U}_{n}\subseteqq\mathcal{U}_{n+1}$;

\item
for every $n\in\mathbb{N}$, $(S_{n},T_{n})$ is a
finitely consistent pair of subsets of $\Phi(\mathcal{U}_{n})$, 
$S_{n}\subseteqq S_{n+1}$, and 
$T_{n}\subseteqq T_{n+1}$. 
\end{enumerate}

\noindent
First, let $\mathcal{U}_{0}=\mathcal{U}$ and  
$(S_{0},T_{0})=(S,T)$. 
Suppose $\mathcal{U}_{i}$ and $(S_{i},T_{i})$ are defined for every $i\leqq n$:

\begin{itemize}
\item
Case $\phi_{n}=\psi_{1}\land\psi_{2}$:
$\mathcal{U}_{n+1}=\mathcal{U}_{n}$. 
If $\psi_{1}\land\psi_{2}\in S_{n}$, 
then $S_{n+1}=S_{n}\cup\{\psi_{1},\psi_{2}\}$ and 
$T_{n+1}=T_{n}$. 
If $\psi_{1}\land\psi_{2}\in T_{n}$, 
then $S_{n+1}=S_{n}$ and  
define $T_{n+1}$ by $T_{n+1}=T_{n}\cup\{\psi_{1}\}$ or 
$T_{n+1}=T_{n}\cup\{\psi_{2}\}$,  
so that $(S_{n+1},T_{n+1})$ is finitely 
consistent.

\item
Case $\phi_{n}=\psi_{1}\supset\psi_{2}$:
$\mathcal{U}_{n+1}=\mathcal{U}_{n}$. 
If $\psi_{1}\supset\psi_{2}\in S_{n}$, 
then define $S_{n+1}$ and $T_{n+1}$ by 
$S_{n+1}=S_{n}$ and  
$T_{n+1}=T_{n}\cup\{\psi_{1}\}$, or 
$S_{n+1}=S_{n}\cup\{\psi_{2}\}$ and  
$T_{n+1}=T_{n}$, 
so that $(S_{n+1},T_{n+1})$ is finitely 
consistent. 
If $\psi_{1}\supset\psi_{2}\in T_{n}$, 
then $S_{n+1}=S_{n}\cup\{\psi_{1}\}$ and 
$T_{n+1}=T_{n}\cup\{\psi_{2}\}$.

\item
Case $\phi_{n}=\neg\psi$:
$\mathcal{U}_{n+1}=\mathcal{U}_{n}$. 
If $\neg\psi\in S_{n}$, 
then 
$S_{n+1}=S_{n}$ and  
$T_{n+1}=T_{n}\cup\{\psi\}$. 
If $\neg\psi\in T_{n}$, 
then $S_{n+1}=S_{n}\cup\{\psi\}$ and 
$T_{n+1}=T_{n}$.

\item
Case $\phi_{n}=\forall x\psi$:
If $\forall x\psi\in S_{n}$, 
then 
$\mathcal{U}_{n+1}=\mathcal{U}_{n}$, 
$S_{n+1}=S_{n}\cup
\{\psi[z/x]\mid z\in\mathcal{U}_{n}\}$,
and  
$T_{n+1}=T_{n}$. 
If $\forall x\psi\in T_{n}$, 
then 
$\mathcal{U}_{n+1}=\mathcal{U}_{n}\cup\{z\}$, 
where $z\in\mathcal{W}\setminus\mathcal{U}_{n}$, 
$S_{n+1}=S_{n}$, and 
$T_{n+1}=T_{n}\cup\{\psi[z/x]\}$.

\item
Otherwise, 
$\mathcal{U}_{n+1}=\mathcal{U}_{n}$ and  
$(S_{n},T_{n})=(S_{n+1},T_{n+1})$. 
\end{itemize}

\noindent
It is clear that the conditions 1 and 2 are satisfied. 
Now, Let 
$$
\mathcal{U}'=\bigcup_{n\in\mathbb{N}}\mathcal{U}_{n}, \ 
S'=\bigcup_{n\in\mathbb{N}}S_{n}, \ 
T'=\bigcup_{n\in\mathbb{N}}T_{n}.
$$

\noindent
Since each formula in $\Phi(\mathcal{W})$ occurs infinitely many times in 
the list $(\phi_{n})_{n\in\mathbb{N}}$, 
$\mathcal{U}'$ and 
$(S',T')$ satisfy
the first part of the 4th condition of Definition \ref{saturated_pair}. 
It is easy to check 
the other conditions are fulfilled. 
\end{proof}

\begin{theorem}\label{s-negbox}
Let $\mathcal{U}$ be a coinfinite subset of $\mathcal{V}$ and 
$(S,T)$ a $\mathcal{U}$-consistent $\propgl$-pair. 
If $\Box\phi\in T$, 
there exists
a coinfinite subset $\mathcal{U}'$ of $\mathcal{V}$ and 
a $\mathcal{U}'$-saturated $\propgl$-pair $(S',T')$ such that 
$\mathcal{U}\subseteqq\mathcal{U}'$, 
$\phi\in T'$, and 
$\Box^{-1}S\cup\Box\Box^{-1}S\subseteqq S'$. 
\end{theorem}

\begin{proof}
Since $(S,T)$ is finitely consistent, 
so is $(\Box^{-1}S\cup\Box\Box^{-1}S,\{\phi\})$. 
Since $(S,T)$ is a $\propgl$-pair, 
$\Box\neg\Diamond^{n}\top\in \Box\Box^{-1}S$ for some $n\in\mathbb{N}$. 
Now, by Theorem \ref{s-saturated}, 
there exists a coinfinite
subset $\mathcal{U}'$ of $\mathcal{V}$ and 
$\mathcal{U}'$-saturated pair $(S',T')$
such that 
$\mathcal{U}\subseteqq\mathcal{U}'$, 
$\phi\in T'$, and 
$\Box^{-1}S\cup\Box\Box^{-1}\subseteqq S'$. 
\end{proof}

\section{Kripke completeness of $\glsysc$}\label{completeness}

In this section, we show that the cut-free fragment $\glsysc$ of $\glsys$
is Kripke complete with respect to $\framebl$. 
The admissibility of the cut-rule follows from the completeness theorem 
and Theorem \ref{soundness}.

\begin{theorem}\label{s-rs}
If $\not\gldbc\gm\yy\dl$, there 
exists a coinfinite subset $\mathcal{U}$ of $\mathcal{V}$ and 
a $\mathcal{U}$-saturated $\propgl$-pair 
$(S,T)$ such that $\gm\subseteqq S$ and $\dl\subseteqq T$.  
\end{theorem}

\begin{proof}
By the rule of boundedness, there exists $n\in\mathbb{N}$ 
such that
$$
\not\gldbc \Box\neg\Diamond^{n}\top,\gm\yy\dl.
$$
Apply Theorem \ref{s-saturated} to 
$\variable(\gm\cup\dl)$ and 
$(\{\Box\neg\Diamond^{n}\top\}\cup\gm,\dl)$. 
\end{proof}

\begin{theorem}
(Kripke completeness of $\glsysc$).
A formula $\phi$
is derivable in $\glsysc$
if and only if 
$\framebl\vl\phi$. 
\end{theorem}

\begin{proof}
We only show the if-part. 
Define a model 
$
\kmodel{M}
=
(W,R,D,I)
$ 
as follows:
\begin{itemize}
\item
$W$ is the set of all triples $(\mathcal{U},S,T)$, where 
$\mathcal{U}$ is a coinfinite subset of $\mathcal{V}$ and 
$(S,T)$ is a $\mathcal{U}$-saturated $\propgl$-pair.

\item
For any $(\mathcal{U},S,T)$ and  
$(\mathcal{U}',S',T')$ in $W$, 
$$
((\mathcal{U},S,T),(\mathcal{U}',S',T'))\in R
\eq
\mbox{
$\mathcal{U}\subseteqq\mathcal{U}'$ and 
$\Box^{-1}S\cup\Box\Box^{-1}S\subseteqq S'$. 
}$$

\item
For any $(\mathcal{U},S,T)\in W$,  
$D_{(\mathcal{U},S,T)}=\mathcal{U}$. 

\item
For any $(\mathcal{U},S,T)\in W$ and any predicate symbol $P$ of arity $n$,
$$
I((\mathcal{U},S,T),P)=
\{
(x_{1},\ldots,x_{n})\in\mathcal{V}^{n}\mid
P(x_{1},\ldots,x_{n})\in S
\}. 
$$
\end{itemize}

\noindent
By definition of $R$,  the frame $(W,R)$ is transitive. 
We claim that 
for any formula $\phi$ and $(\mathcal{U},S,T)\in W$,
$$
\phi\in S
\thn 
\kmodel{M},(\mathcal{U},S,T)\vl\phi,\hspace{10pt}
\phi\in T
\thn
\kmodel{M},(\mathcal{U},S,T)\not\vl\phi
. 
$$
We show the claim only for the cases of $\phi=P(x_{1},\ldots,x_{n})$,  $\forall x\psi(x)$, 
and $\Box\psi$:
\begin{itemize}

\item
Case $\phi=P(x_{1},\ldots,x_{n})$: 
By definitions of $I$ and $\vl$, 
\begin{eqnarray*}
P(x_{1},\ldots,x_{n})\in S
&\eq&
(x_{1},\ldots,x_{n})\in I((\mathcal{U},S,T),P)\\
&\eq&
\kmodel{M},(\mathcal{U},S,T)\vl P(x_{1},\ldots,x_{n}). 
\end{eqnarray*} 
Since $(S,T)$ is finitely consistent, 
\begin{eqnarray*}
P(x_{1},\ldots,x_{n})\in T
&\thn&
P(x_{1},\ldots,x_{n})\not\in S\\
&\eq&
(x_{1},\ldots,x_{n})\not\in I((\mathcal{U},S,T),P)\\
&\eq&
\kmodel{M},(\mathcal{U},S,T)\not\vl P(x_{1},\ldots,x_{n}). 
\end{eqnarray*}

\item
Case $\phi=\forall x\psi(x)$:
If
$\forall x\psi(x)\in S$,  
then
$\psi(z)\in S$ for any $z\in \mathcal{U}$, 
since 
$(S,T)$ is $\mathcal{U}$-saturated. 
Hence, 
by induction hypothesis, 
$\kmodel{M},(\mathcal{U},S,T)\vl \psi(z)$ 
for any $u\in D_{(\mathcal{U},S,T)}$. 
If
$
\forall x\psi(x)\in T,  
$
then, 
$\psi(z)\in T$ for some $z\in \mathcal{U}$,  
since 
$(S,T)$ is $\mathcal{U}$-saturated. 
By induction hypothesis, 
$\kmodel{M},(\mathcal{U},S,T)\not\vl \psi(z)$
for some $z\in D_{(\mathcal{U},S,T)}$.

\item
Case $\phi=\Box\psi$:
Suppose $\Box\psi\in S$ and $((\mathcal{U},S,T),(\mathcal{U}',S',T'))\in R$. 
Then, 
$\psi\in S'$ 
by definition of $R$.  
By induction hypothesis, 
$\kmodel{M},(\mathcal{U}',S',T')\vl \psi$.
Suppose $\Box\psi\in T$. 
Then, 
by Theorem \ref{s-negbox}, 
there exists a coinfinite subset $\mathcal{U}'$ of $\mathcal{V}$ and 
a $\mathcal{U}'$-saturated $\propgl$-pair $(S',T')$ such that 
$\mathcal{U}\subseteqq\mathcal{U}'$, 
$\phi\in T'$, and 
$\Box^{-1}S\cup\Box\Box^{-1}\subseteqq S'$. 
Then, $(\mathcal{U}',S',T')\in W$,  $((\mathcal{U},S,T),(\mathcal{U}',S',T'))\in R$, 
and, by induction hypothesis,
$\kmodel{M},(\mathcal{U}',S',T')\not\vl \psi$.  
\end{itemize}

\noindent
This complete the proof of the claim. By using the claim and 
Lemma \ref{boundedness},  $(W,R)\in\framebl$. 
Now, suppose $\not\gldbc\gm\yy\dl$. Then, by Theorem \ref{s-rs},  
there exists $(\mathcal{U},S,T)\in W$ such that $\gm\subseteqq S$ and $\dl\subseteqq T$. 
Hence, $\kmodel{M},(\mathcal{U},S,T)\not\vl\gm\yy\dl$. 
\end{proof}

\newcommand{\noop}[1]{}


\begin{thebibliography}{10}

\bibitem{art-dzh90}
S.~Artemov and G.~Dzhaparidze.
\newblock Finite Kripke models and predicate logics of provability.
\newblock {\em The Journal of Symbolic Logic}, 55:1090--1098, 1990.

\bibitem{avr84}
A.~Avron.
\newblock On modal systems having arithmetical interpretations.
\newblock {\em The Journal of Symbolic Logic}, 49:935--942, 1984.

\bibitem{bls93}
G.~Boolos.
\newblock {\em The logic of provability}.
\newblock Cambridge University Press, 1993.

\bibitem{gld93}
R.~Goldblatt.
\newblock {\em Mathematics of Modality}, volume~43 of {\em CSLI Lecture Notes}.
\newblock CSLI Publications, 1993.

\bibitem{gor-rmn08}
R.~Gor\'{e} and R.~Ramanayake.
\newblock Valentini's cut-elimination for provability logic resolved.
\newblock In C.~Areces and R.~Goldblatt, editors, {\em Advances in Modal
  Logic}, volume~7, pages 67--86. CSLI Publications, 2008.

\bibitem{lvn81}
D. Leivant.
\newblock On the proof theory of modal logic for arithmetic provability.
\newblock {\em The Journal of Symbolic Logic}, 46:531--538, 1981.

\bibitem{mnt84}
F.~Montagna.
\newblock The predicate modal logic of provability.
\newblock {\em Notre Dame Journal of Formal Logic}, 25:179--189, 1984.

\bibitem{sch-trl14}
Y.~Schwarz and G.~Tourlakis.
\newblock On the proof-theory of a first-order extension of {GL}.
\newblock {\em Logic and Logical Philosophy}, 23:329--363, 2014.

\bibitem{sgb94}
K.~Segerberg.
\newblock A model existence theorem in infinitary propositional modal logic.
\newblock {\em Journal of Philosophical Logic}, 23:337--367, 1994.

\bibitem{tnknoncpt}
Y.~Tanaka.
\newblock Model existence in non-compact modal logic.
\newblock {\em Studia Logica}, 67:61--73, 2001.

\bibitem{vln83}
S.~Valentini.
\newblock The modal logic of provability.
\newblock {\em Journal of Philosophical Logic}, 12:471--476, 1983.

\bibitem{yvr01}
R.~E.~Yavorsky.
\newblock On arithmetical completeness of first order logics of provability.
\newblock In F.~Wolter, H.~Wansin, and M.~Zakharyaschev, editors, {\em Advances
  in Modal Logic}, volume~3, pages 1--16. CSLI Publications, 2001.

\end{thebibliography}
\end{document}